\documentclass[article,12pt]{amsart}%
\usepackage{amscd}
\usepackage{bbm}
\usepackage{dsfont}
\usepackage{enumerate}
\usepackage{latexsym}
\usepackage{srcltx}
\usepackage{amsfonts}
\usepackage{amssymb}
\usepackage{datetime}
\usepackage{setspace}
\usepackage{enumerate}
\usepackage{amsthm}
\usepackage{graphicx}
\usepackage{epstopdf}
\usepackage{amsmath}%
\providecommand{\U}[1]{\protect\rule{.1in}{.1in}}
\newtheorem{theorem}{Theorem}[section]
\newtheorem{proposition}[theorem]{Proposition}
\newtheorem{lemma}[theorem]{Lemma}
\newtheorem{corollary}[theorem]{Corollary}

\theoremstyle{definition}
\newtheorem{example}[theorem]{Example}
\newtheorem{definition}[theorem]{Definition}

\newtheorem{remark} [theorem] {Remark}

\input epsf
\begin{document}
\title{ Weak sequential completeness in \\ Banach $C(K)$-modules of finite multiplicity.}
\author{Arkady Kitover}
\address{Community College of Philadelphia, 1700 Spring Garden St., Philadelphia, PA
19154, USA}
\email{akitover@ccp.edu}
\author{Mehmet Orhon}
\address{Department of Mathematics and Statistics, University of New Hampshire, Durham,
NH 03824, USA}
\email{mo@unh.edu}
\date{\today}
\maketitle

\begin{abstract}
A well known result of Lozanovsky states that a Banach lattice is weakly
sequentially complete if and only if it does not contain a copy of $c_{0}$. In
the current paper we extend this result to the class of Banach $C(K)$-modules
of finite multiplicity and, as a special case, to finitely generated Banach
$C(K)$-modules. Moreover, we prove that such a module is weakly sequentially
complete if and only if each cyclic subspace of the module is weakly
sequentially complete.

\end{abstract}

\markboth{Arkady Kitover and Mehmet Orhon}{Weak sequential completeness}

\section{Introduction}

It is well known that some important properties of Banach spaces (reflexivity,
weak sequential completeness, et cetera) can be more readily studied if we
restrict our attention to the class of Banach lattices. In particular, it was
proved by Lozanovsky in~\cite{Loz} that a Banach lattice is reflexive if and
only if it does not contain a subspace \footnote{By subspace we always mean
closed subspace.} isomorphic to either $c_{0}$ or $l_{1}$, and that it is
weakly sequentially complete if and only if it does not contain a subspace
isomorphic to $c_{0}$. The above results cannot be extended to the class of
all Banach spaces: the famous James space~\cite{Ja} does not contain either
$l^{1}$ or $c_{0}$ but is neither reflexive nor (see~\cite[page 34]{LT2})
weakly sequentially complete.

Nevertheless, Lozanovsky's results remain true for arbitrary subspaces of
Banach lattices with order continuous norm or complemented subspaces of
arbitrary Banach lattices, see~\cite{Tz} and~\cite[Theorem 1.c.7, page
37]{LT2}. We consider similar problems for two other subclasses of the class
of Banach spaces, namely the class of finitely generated Banach $C(K)$-modules
(see Definition~\ref{d2} below) and the class of Banach $C(K)$-modules of
finite multiplicity (Definition~\ref{d8}). \footnote{It is worth emphasizing
that under the condition that every cyclic subspace is a $KB$-space the former
subclass is strictly smaller than the latter.} Because cyclic subspaces of
Banach $C(K)$-modules can be represented in a natural way as Banach lattices
we have reason to believe that many ``nice'' properties of Banach lattices
would carry over to Banach $C(K)$-modules of finite multiplicity. In particular,
the authors proved in~\cite{KO} that a Banach $C(K)$-module of finite
multiplicity is reflexive if and only if it does not contain a copy of either
$l^{1}$ or $c_{0}$. The goal of the current paper is to prove that
Lozanovsky's characterization of weak sequential completeness of Banach
lattices also remains true for such modules.

\bigskip

\section{Preliminaries}

\bigskip All the linear spaces will be considered either over the field of
real numbers $\mathds{R}$ or the field of complex numbers $\mathds{C}$. If $X$
is a Banach space we will denote its Banach dual by $X^{\star}$.

Let us recall some definitions.

\begin{definition}
\label{d1} Let $K$ be a compact Hausdorff space and $X$ be a Banach space. We
say that $X$ is a Banach $C(K)$-module if there is a continuous unital
homomorphism $m$ of $C(K)$ into the algebra $L(X)$ of all bounded linear
operators on $X$.
\end{definition}

\begin{remark}
\label{r1} Because $\ker{m}$ is a closed ideal in $C(K)$ by considering, if
needed, $C(\tilde{K}) = C(K)/\ker{m}$ we can and will assume without loss of
generality that $m$ is a contractive homomorphism (see~\cite{KO}) and $\ker{m}
= 0$. Then (see~\cite[Lemma 2 (2)]{HO}) $m$ is an isometry. Moreover, when it
does not cause any ambiguity we will identify $f \in C(K)$ and $mf \in L(X)$.
\end{remark}

\begin{definition}
\label{d2} Let $X$ be a Banach $C(K)$-module and $x \in X$. We introduce the
\textit{cyclic} subspace $X(x)$ of $X$ as $X(x) = cl\{fx : f \in C(K)\}$.
\end{definition}

The following proposition was proved in~\cite{Ve} (see also~\cite{Ra}) in the case
when the compact space $K$ is extremally disconnected and announced for an
arbitrary compact Hausdorff space $K$ in~\cite{Ka}. It follows as a special
case from~\cite[Lemma 2 (2)]{HO}.

\begin{proposition}
\label{p1} Let $X$ be a Banach $C(K)$-module, $x \in X$, and $X(x)$ be the
corresponding cyclic subspace. Then, endowed with the cone $X(x)_{+} = cl \{fx
: f \in C(K), f \geq0\}$ and the norm inherited from $X$, $X(x)$ is a Banach
lattice with positive quasi-interior point $x$.
\end{proposition}

Our next proposition follows from Theorem 1 (3) in~\cite{Or}

\begin{proposition}
\label{p2} The center $Z(X(x))$ of the Banach lattice $X(x)$ is isometrically
isomorphic to the weak operator closure of $m(C(K))$ in $L(X(x))$.
\end{proposition}

Now we can introduce one of the two main objects of interest in the current paper.

\begin{definition}
\label{d3} Let $X$ be a Banach $C(K)$-module. We say that $X$ is finitely
generated if there are an $n \in\mathds{N}$ and $x_{1}, \ldots, x_{n} \in X$
such that the set $\sum\limits_{i=1}^{n} X(x_{i})$ is dense in $X$.
\end{definition}

Before we proceed with the statement and the proof of our main results
(Theorem~\ref{t1} and Theorem~\ref{t2}) we need to introduce a few more
definitions and recall a couple of results.

\begin{definition}
\label{d5} Let $X$ be a Banach space and $\mathcal{B}$ be a Boolean algebra of
projections on $X$. The algebra $\mathcal{B}$ is called Bade complete
(see~\cite[XVII.3.4, p.2197]{DS} and~\cite[V.3, p.315]{Sch}) if

(1) $\mathcal{B}$ is a complete Boolean algebra.

(2) Let $\{\chi_{\gamma}\}_{\gamma\in\Gamma}$ be an increasing net in
$\mathcal{B}$, $\chi$ be the supremum of this net, and $x\in X$. Then the net
$\{\chi_{\gamma}x\}$ converges to $\chi x$ in norm in $X$.
\end{definition}

The following result was proved in~\cite[Theorem 3 and Remark 5]{Or1}.

\begin{theorem}
\label{t0}

Let $X$ be a Banach space, $K$ be a compact Hausdorff space, and $m$ be an
isometric embedding of $C(K)$ into $L(X)$. Assume that no cyclic subspace of $X$ contains a
copy of $c_{0}$. Then

(1) The closure of $m(C(K))$ in the weak operator topology is isometrically
isomorphic to $C(Q)$ where $Q$ is a hyperstonian space.

(2) The algebra of all idempotents in $C(Q)$ is a Bade complete Boolean
algebra of projections on $X$.
\end{theorem}

\begin{remark}
\label{r2} In virtue of Theorem~\ref{t0} from now on, in this section, we will
assume that the compact Hausdorff space $K$ is hyperstonian and that $\mathcal {B}$, the algebra of all idempotents in $C(K)$, is a Bade complete Boolean algebra of projections on $X$.
\end{remark}

\begin{remark}
\label{r3} Assume conditions of Remark~\ref{r2}. Because $K$ is hyperstonian
$C(K)$ is a dual Banach space (see e.g.~\cite[Theorem 9.3, page 122]{Sch}).
Let us denote its predual by $C(K)_{\star}$. Because $C(K)_{\star}$ is an
$AL$-space it is isometrically isomorphic to a band in its second dual, i.e.,
to a band in $C(K)^{\star}$ (see~\cite[Page 92]{Sch}).

Next notice (see~\cite{Or}) that the center $Z(C(K)^{\star})$ of $C(K)^{\star
}$ is isometrically isomorphic to $C(K)^{\star\star}$. We will denote by $p$
the idempotent in $C(K)^{\star\star}$ corresponding to the band projection in
$C(K)^{\star}$ onto the band $C(K)_{\star}$.
\end{remark}

\begin{remark}
\label{r6} Assume conditions of Remark~\ref{r2} and assume also that $X$ is
cyclic. Then by~\cite{Ve} (see also~\cite[V.3]{Sch}), $X$ can be represented
as a Banach lattice with order continuous norm and a positive quasi-interior
point such that $\mathcal{B}$ coincides with the algebra of all band
projections on $X$.
\end{remark}

\begin{remark}
\label{r4} Consider $X^{\star}$ as a Banach $C(K)$-module as follows. Let
$m^{\star}:C(K)\rightarrow L(X^{\star})$ be the isometric unital algebra and
lattice homomorphism that describes this module structure. That is, for each
$a\in C(K)$, let $m^{\star}(a)$ be the Banach space adjoint of $a$ acting as a
bounded operator on $X$. Hence $ax^{\star}(x)=m^{\star}(a)(x^{\star
})(x)=x^{\star}(ax)$ for all $x\in X$ and $x^{\star}\in X^{\star}$.

Furthermore, by means of the Arens extension procedure~\cite{Ar}, we will
consider $X^{\star\star}$ as a Banach $C(K)^{\star\star}$-module. Namely, for
each $x^{\star}\in X^{\star}$, $x^{\star\star}\in X^{\star\star}$, $a\in
C(K)$, and $a^{\star\star}\in C(K)^{\star\star}$, we define $\mu_{x^{\star
},x^{\star\star}}\in C(K)^{\star}$ and $a^{\star\star}\cdot x^{\star\star}\in
X^{\star\star}$ as follows%
\begin{align*}
\mu_{x^{\star},x^{\star\star}}(a)  &  =x^{\star\star}(ax^{\star}),\\
a^{\star\star}\cdot x^{\star\star}(x^{\star})  &  =a^{\star\star}%
(\mu_{x^{\star},x^{\star\star}}).
\end{align*}

\end{remark}

We will list the relevant easily checked properties of the above definitions
in the following lemma.

\begin{lemma}
\label{l1} Let $x^{\star}\in X^{\star},$ $x^{\star\star}\in X^{\star\star}$,
and $a^{\star\star},b^{\star\star}\in C(K)^{\star\star}.$ Then

\begin{enumerate}
\item $a^{\star\star}\mu_{x^{\star},x^{\star\star}}=\mu_{x^{\star}%
,a^{\star\star}\cdot x^{\star\star}}$ ;

\item $(a^{\star\star}b^{\star\star})\cdot x^{\star\star}=a^{\star\star}%
\cdot(b^{\star\star}\cdot x^{\star\star})$ ;

\item We have $x^{\star\star}\in p\cdot X^{\star\star}$ if and only if
$\mu_{x^{\star},x^{\star\star}}\in C(K)_{\ast}$ for all $x^{\star}$.
\end{enumerate}
\end{lemma}

\begin{proof}
(1) Given $a^{\star\star}$ in $C(K)^{\star\star}$ let $\{a_{\alpha}\}$ be a
net in $C(K)$ that converges to $a^{\star\star}$ in the $\sigma(C(K)^{\star
\star},C(K)^{\star})$-topology. Then, for any $a\in C(K),$ we have%
\begin{align*}
a^{\star\star}\mu_{x^{\star},x^{\star\star}}(a)  &  =\underset{}{a^{\star
\star}(a\mu_{x^{\star},x^{\star\star}})=}\underset{\alpha}{\lim}\mu_{x^{\star
},x^{\star\star}}(a_{\alpha}a)=\underset{\alpha}{\lim}x^{\star\star}%
(a_{\alpha}ax^{\star});\\
\mu_{x^{\star},a^{\star\star}\cdot x^{\star\star}}(a)  &  =a^{\star\star}\cdot
x^{\star\star}(ax^{\star})=a^{\star\star}(\mu_{ax^{\star},x^{\star\star}%
})=\underset{\alpha}{\lim}\mu_{ax^{\star},x^{\star\star}}(a_{\alpha
})=\underset{\alpha}{\lim}x^{\star\star}(a_{\alpha}ax^{\star}).
\end{align*}

(2) Applying part (1) when necessary we have for any $a\in C(K)$,%
\begin{align*}
(a^{\star\star}b^{\star\star})\cdot x^{\star\star}(ax^{\star})  &
=\mu_{x^{\star},(a^{\star\star}b^{\star\star})\cdot x^{\star\star}%
}(a)=(a^{\star\star}b^{\star\star})\mu_{x^{\star},x^{\star\star}}%
(a)=a^{\star\star}(b^{\star\star}\mu_{x^{\star},x^{\star\star}})(a)\\
&  =\mu_{x^{\star},a^{\star\star}\cdot(b^{\star\star}\cdot x^{\star\star}%
)}(a)=a^{\star\star}\cdot(b^{\star\star}\cdot x^{\star\star})(ax^{\star}).
\end{align*}

(3) Suppose $x^{\star\star}\in p\cdot X^{\star\star}$. Then $x^{\star\star
}=p\cdot x^{\star\star}.$ Hence, by part (1),%
\[
\mu_{x^{\star},x^{\star\star}}=\mu_{x^{\star},p\cdot x^{\star\star}}%
=p\mu_{x^{\star},x^{\star\star}}\in C(K)_{\ast}.
\]
Conversely, suppose $\mu_{x^{\star},x^{\star\star}}\in C(K)_{\ast}$ for all
$x^{\star}\in X^{\star}$. Then $\mu_{x^{\star},x^{\star\star}}=p\mu_{x^{\star
},x^{\star\star}}=\mu_{x^{\star},p\cdot x^{\star\star}}$ for all $x^{\star}\in
X^{\star}$. It follows that $x^{\star\star}=p\cdot x^{\star\star}$.
\end{proof}

\begin{lemma}
\label{l2} Let $X$ be a Banach lattice. Then $p\cdot X^{\star\star}$ is equal
to $(X^{\star})_{n}^{\star}$.
\end{lemma}

\begin{proof}
Let $X$ be a Banach lattice. Then $Z(X^{\star})$, the ideal center of
$X^{\star}$, is equal to $C(K)$ where $K$ is hyperstonian. Since $X^{\star}$
is Dedekind complete, $Z(X^{\star})$ is topologically full in the sense of
Wickstead~\cite{Wi}. That is for any $x^{\star}\in X_{+}^{\star}$, one has
that $cl(Z(X^{\star})x^{\star})$ is the closed ideal generated by $x^{\star}$.
Then, by considering $X^{\star}$ as a $C(K)$-module and applying the Arens
extension process described above we obtain (see~\cite[Corollary 1]{Or}) that
the ideal center $Z(X^{\star\star})$ of $X^{\star\star}$ is equal to a band in
$C(K)^{\star\star}$. It is easy to see from Lemma \ref{l1} part (3) that in
$X^{\star\star}$, the set $p\cdot X^{\star\star}$ is precisely the band of
order continuous linear functionals on $X^{\star}$. That is $p\cdot
X^{\star\star}=(X^{\star})_{n}^{\star}$.
\end{proof}

\begin{remark}
\label{r5} In the case of Riesz spaces, the statement of Lemma~\ref{l2}
remains true for the order bidual of the Riesz space. We refer the interested
reader to Corollary 6 in~\cite{AO}.
\end{remark}

Our next lemma shows that the property $p \cdot X^{\star\star} = X$ is a
three-space property.

\begin{lemma}
\label{l3} Suppose $K$ is hyperstonian and $X$ is a Banach $C(K)$-module such
that $\ \mathcal{B}$, the idempotents in $C(K)$, is a Bade complete Boolean
algebra of projections on $X$.

\begin{enumerate}
\item Let $Y$ be a closed submodule of $X$ such that $p\cdot Y^{\star\star}=Y$
and $p\cdot(X/Y)^{\star\star}=X/Y$. Then $p\cdot X^{\star\star}=X$.

\item Let $p\cdot X^{\star\star}=X$. Then for any closed submodule $Y$ of $X$
we have $p\cdot Y^{\star\star}=Y$ and $p\cdot(X/Y)^{\star\star}=X/Y$.
\end{enumerate}
\end{lemma}

\begin{proof}
(1) Let $Y$ be a closed submodule of $X$. Then $\ \mathcal{B}$ is a Bade
complete Boolean algebra of projections on both $Y$ and $X/Y$
(see~\cite[Lemma 1]{KO}). Hence it makes sense to consider the hypothesis in
the statement of part (1) of the lemma and assume that it holds for these two
spaces. It is familiar from standard duality that $Y^{\star\star}%
=Y^{oo}\subset X^{\star\star}$ and $(X/Y)^{\star\star}=X^{\star\star}/Y^{oo}$
where for a subspace $Y\subset X$, we let $Y^{o}$ denote the polar (or the
annihilator) of $Y$ in $X^{\star}$ and let $Y^{oo}$ denote the polar of
$Y^{o}$ in $X^{\star\star}$. Consider $x^{\star\star}\in X^{\star\star}$. Let
$[x^{\star\star}]=x^{\star\star}+Y^{oo}\in X^{\star\star}/Y^{oo}$. Then
$p\cdot(X/Y)^{\star\star}=X/Y$ means that, as a subset of $X^{\star\star},$
$p\cdot\lbrack x^{\star\star}]=[p\cdot x^{\star\star}]=p\cdot x^{\star\star
}+Y^{oo}$ has non-empty intersection with $X$. That is, there is $x\in X$ and
$y^{\star\star}\in Y^{oo}$ such that $p\cdot x^{\star\star}+y^{\star\star}=x$
(*). Now $p\cdot Y^{oo}=Y$ means that $p\cdot y^{\star\star}=y$ for some $y\in
Y$ and we also have that $p\cdot x=x$ trivially. Hence we apply $p$ to both
sides of the equality (*) to obtain $p\cdot x^{\star\star}=x-y$ for some $y\in
Y$ and $x\in X$.

(2) Assume $p\cdot X^{\star\star}=X$ and $Y$ is a closed submodule of $X$. It
is familiar that $Y=X\cap Y^{oo}$ by the bipolar theorem. Let $y^{\star\star
}\in Y^{oo}$, then $p\cdot y^{\star\star}\in X\cap Y^{oo}=Y$. Hence $p\cdot
Y^{\star\star}=Y$. Now take $x^{\star\star}\in X^{\star\star}$ and consider
$[x^{\star\star}]\in X^{\star\star}/Y^{oo}$. Since $p\cdot x^{\star\star}=x$
for some $x\in X$, we have $p\cdot\lbrack x^{\star\star}]=[p\cdot
x^{\star\star}]=[x]$. Also, given $\varepsilon>0$, there is $y^{\star\star}\in
Y^{oo}$ such that $||p\cdot x^{\star\star}+y^{\star\star}||<||p\cdot\lbrack
x^{\star\star}]||(1-\varepsilon)$. Since $p\cdot y^{\star\star}=y\in Y$ for
some $y\in Y$, we have $p\cdot(p\cdot x^{\star\star}+y^{\star\star})=x+y$.
Since $||p||=1$, we have, when $p\cdot x^{\star\star}=x$, $||p\cdot\lbrack
x^{\star\star}]||_{(X/Y)^{\star\star}}=||[x]||_{X/Y}$. Hence $p\cdot
(X/Y)^{\star\star}=X/Y$.
\end{proof}

The next lemma is a special case of Theorem 6 in~\cite[Page 297]{KA}.

\begin{lemma}
\label{l4} Let $X$ be a Banach lattice and let $\{x_{n}^{\star}\}$ be a
sequence of positive, increasing and norm bounded elements of $X^{\star}$.
Suppose $M>0$ is the supremum of the norms of the functionals in the sequence.
Then the sequence has a least upper bound $x^{\star}\in X_{+}^{\star}$ with
norm equal $M$.
\end{lemma}

We will need the following complement to Lozanovsky's result that follows from
Theorem 2.4.12 in~\cite{MN}.

\begin{theorem}
\label{t00} Let $X$ be a Banach lattice. The following conditions are equivalent.

(1) $X$ is weakly sequentially complete.

(2) $X$ does not contain a sublattice that is lattice isomorphic to $c_{0}$.
\end{theorem}

The following lemma will play a crucial role in the proof of Theorem~\ref{t1}.

\begin{lemma}
\label{l5} Suppose $K$ is hyperstonian and $X$ is a finitely generated Banach
$C(K)$-module such that the algebra $\ \mathcal{B}$ of the idempotents in
$C(K)$, is a Bade complete Boolean algebra of projections on $X$. If no cyclic
subspace of $X$ contains a copy of $c_{0},$ then $p\cdot X^{\star\star}=X.$
\end{lemma}

\begin{proof}
We will prove the result by induction on the number of generators of $X.$
Suppose $X$ is cyclic. Then by Remark~\ref{r6}, $X$ can be represented as a
Banach lattice with order continuous norm such that $\ \mathcal{B}$
corresponds to the band projections on $X.$ Since $X$ is cyclic $X$ does not
contain a copy of $c_{0}.$ Then, by the well known result for Banach lattices
(see e.g.~\cite[Theorem 2.4.12]{MN}), $X$ is a KB-space and $X=(X^{\star}%
)_{n}^{\star}.$ But by Lemma~\ref{l2} we have that $p\cdot X^{\star\star
}=(X^{\star})_{n}^{\star}$ for Banach lattices. Therefore $p\cdot
X^{\star\star}=X$ as required. Now suppose whenever $X$ has $r$ generators (
$r\geq1$ ) and no cyclic subspace of $X$ contains a copy of $c_{0},$ we have
that $p\cdot X^{\star\star}=X$. Suppose that $X$ has $r+1$ generators. Let
$\{x_{0},x_{1},\cdots,x_{r}\}$ be a set of generators for $X.$ Let
$Y=X(x_{1},x_{2},\cdots,x_{r}).$ Then, since no cyclic subspace of $Y$
contains a copy of $c_{0}$, we have that $p\cdot Y^{\star\star}=Y$ or $p\cdot
Y^{oo}=Y$ since $Y^{\star\star}=Y^{oo}\subset X^{\star\star}.$ Now consider
$X/Y=X/Y([x_{0}]).$ Since $\ \mathcal{B}$ is Bade complete on $X$, it is also
Bade complete on $X/Y$ (~\cite[Lemma1]{KO}). Also since $X/Y$ is cyclic it may
be represented as a Banach lattice which, by means of the previous sentence,
may be assumed to have an order continuous norm and such that $\ \mathcal{B}$
corresponds to the band projections of the Banach lattice. If $X/Y$ contains
no copy of $c_{0}$, then, as in the case when $X$ is cyclic, we may claim that
$p\cdot(X/Y)^{\star\star}=X/Y.$ So if this is not the case, $X/Y$ must contain
a copy of $c_{0}.$ In fact as remarked before since $X/Y$ is a Banach lattice
it must contain a copy of $c_{0}$ that is lattice isomorphic to a closed
sublattice of $X/Y.$ This is equivalent to the statement that there exists a
sequence $\{u_{n}\}$ in $X$ such that $\{[u_{n}]\}$ is a positive disjoint
sequence in $X/Y$ and there exists constants $0<d<D$ with
\[
d\leq||[u_{n}]||\,\text{ and \thinspace}||[u_{1}]+[u_{2}]+\cdots+[u_{n}]||\leq
D
\]
for each $n$ (see~\cite[Lemma 2.3.10]{MN}). Let $\{e_{n}\}$ be a sequence in
$\ \mathcal{B}$ such that each $e_{n}$ is the band projection onto the band
generated by $[u_{n}]$ in $X/Y.$ Since the sequence $\{[u_{n}]\}$ is disjoint
in the Banach lattice $X/Y,$ the members of the sequence of band projections
$\{e_{n}\}$ are also pairwise disjoint as idempotents in $C(K).$ Since
$e_{n}[u_{n}]=[e_{n}u_{n}]=[u_{n}]$ for each $n,$ we may assume without loss
of generality that $e_{n}u_{n}=u_{n}$. Let $z_{n}=u_{1}+u_{2}+\cdots+u_{n}\in
X$ and $\chi_{n}=e_{1}+e_{2}+\cdots+e_{n}\in\ \mathcal{B}$ for each $n.$ Then
$\{[z_{n}]\}$ is a positive, increasing and norm bounded sequence in $X/Y.$ By
Lemma~\ref{l4}, there exists $z\in X^{\star\star}$ such that $[z]=\underset
{n}{\sup}[z_{n}]$. Note that we consider $X/Y\subset((X/Y)^{\star})_{n}%
^{\star}$. Since $((X/Y)^{\star})_{n}^{\star}$ is a band in $(X/Y)^{\star
\star},$ we have that $[z]\in((X/Y)^{\star})_{n}^{\star}.$ Then by
Lemma~\ref{l2}, we have that $p\cdot\lbrack z]=[z].$ Therefore without loss of
generality we take $p\cdot z=z.$ The definition of $\{z_{n}\}$ implies that
$\chi_{n}z_{m}=z_{n}$ for all $n\leq m.$ Therefore $\chi_{n}\cdot\lbrack
z]=[z_{n}]$ for all $n.$ This means that there exists $y_{n}^{\star\star}\in
Y^{oo}$ such that $\chi_{n}\cdot z+y_{n}^{\star\star}=z_{n}$ for each $n.$ By
induction hypothesis, we have $p\cdot y_{n}^{\star\star}=y_{n}$ for some
$y_{n}\in Y$ and since $z_{n}\in X,$ we have $p\cdot z_{n}=z_{n}$ for each
$n$. Therefore if we apply the projection $p$ to both sides of the above
equality, we obtain $\chi_{n}\cdot z+y_{n}=z_{n}$ for each $n.$ Now apply
$e_{n}$ to the last equality to obtain
\[
e_{n}\cdot z=u_{n}-e_{n}y_{n}\in X
\]
for each $n.$ Note that the sequence $\{e_{n}\cdot z\}$ in $X$ satisfies%
\[
\delta\leq||[u_{n}]||\leq||e_{n}\cdot z||\text{\thinspace\ and \thinspace
}||e_{1}\cdot z+e_{2}\cdot z+\cdots+e_{n}\cdot z||=||\chi_{n}\cdot
z||\leq||z||
\]
for each $n.$ Let $w=\sum\frac{1}{2^{n}}e_{n}\cdot z$ in $X$ and consider the
cyclic subspace $X(w).$ Clearly we may represent $X(w)$ as a Banach lattice
with order continuous norm and positive quasi-interior point $w$. In virtue of
Lemma 1 in~\cite{KO} we may assume that $\mathcal{B}$ corresponds to the
algebra of band projections on the Banach lattice $X(w)$. Then $e_{n}%
w=\frac{1}{2^{n}}e_{n}\cdot z$ for each $n.$ Hence $\{e_{n}\cdot z\}$ is a
disjoint positive sequence in the Banach lattice $X(w).$ The displayed
conditions above indicates that the sublattice generated by $\{e_{n}\cdot z\}$
in $X(w)$ is lattice isomorphic to $c_{0}$ ~\cite[Lemma 2.3.10]{MN}. But this
is a contradiction, because we assumed that no cyclic subspace of $X$ contains
a copy of $c_{0}.$ Thus we see that $p\cdot(X/Y)^{\star\star}=X/Y$ as well as
$p\cdot Y^{\star\star}=Y.$ Finally Lemma~\ref{l3} part $(1)$ implies that
$p\cdot X^{\star\star}=X$ as required.
\end{proof}

\section{The main results.}

\subsection{Finitely generated $C(K)$-modules.}

\begin{theorem}
\label{t1} Let $X$ be a finitely generated Banach $C(K)$-module. Then the following
conditions are equivalent.

\begin{enumerate}
\item $X$ is weakly sequentially complete.

\item $X$ does not contain a copy of $c_{0}.$

\item No cyclic subspace of $X$ contains a copy of $c_{0}.$

\item Each cyclic subspace of $X$ is weakly sequentially complete.
\end{enumerate}
\end{theorem}

\begin{proof}
It is evident that we have $(1)\Rightarrow(2)$ and $(2)\Rightarrow(3)$. The
equivalence $(3)\Leftrightarrow(4)$ follows from Theorem~\ref{t00} and
Lozanovsky's result in~\cite{Loz}. Hence the proof will be completed if we
show that $(3)\Rightarrow(1).$ Suppose no cyclic subspace of $X$ contains a
copy of $c_{0}.$ Then Theorem~\ref{t0} guarantees that the weak operator
closure of $C(K)$ in $L(X)$ is the norm closure of the linear span of a Bade
complete Boolean algebra of projections on $X.$ Since the closed submodules of
$X$ are the same with respect to either algebra of operators (i.e., either
with respect to $C(K)$ or with respect to its weak operator closure in $L(X)$
) we may assume that $C(K)$ is weak operator closed in $L(X).$ Then (see
Theorem~\ref{t0}) $K$ is hyperstonian and $\ \mathcal{B}$, the algebra of
idempotents in $C(K),$ is the Bade complete Boolean algebra of projections
that generate $C(K).$ Then $(3)$ and Lemma \ref{l5} imply that $p\cdot
X^{\star\star}=X.$ To save space we will say that $X$ satisfies $(\star)$ if
$X$ is a $C(K)$-module with hyperstonian $K$ such that the idempotents in
$C(K)$ form a Bade complete Boolean algebra of projections on $X$ and $p\cdot
X^{\star\star} = X$. Hence $(3)$ implies that $X$ satisfies $(\star)$. Once
again we will use induction on the number of generators of $X.$ If $X$ is
cyclic then it may be represented as a Banach lattice with order continuous
norm such that the band projections correspond to $\ \mathcal{B}$. By
Lemma~\ref{l2}, we have $p\cdot X^{\star\star}=(X^{\star})_{n}^{\star}$.
Therefore $p\cdot X^{\star\star}=X$ implies $X=(X^{\star})_{n}^{\star}$. Hence
by Theorem~\ref{t00}, $X$ is weakly sequentially complete (see
also~\cite[Theorem 8, page 297]{KA}) . Now, as induction hypothesis, we
suppose that whenever $X$ is a $C(K)$-module with at most $r$ generators (
$r\geq1$ ) and $X$ satisfies the condition $(\star)$ then $X$ is weakly
sequentially complete. Suppose $X$ has $r+1$ generators and satisfies
$(\star)$. Let $\{x_{0},x_{1},\cdots,x_{r}\}$ be a set of generators for $X$
and let $Y=X(x_{1},x_{2},\cdots,x_{r}).$ Then, by Lemma~\ref{l3} part $(2),$
both $Y$ and $X/Y$ satisfy the condition $(\star)$. Moreover $Y$ has $r$
generators and $X/Y=X/Y([x_{0}])$ is cyclic. Therefore by the induction
hypothesis $Y$ and $X/Y$ are both weakly sequentially complete. Then the
three-space property of weak sequential completeness (see~\cite[Theorem
4.7.a., page 122]{CG}) implies that $X$ is weakly sequentially complete.
\end{proof}

We conclude this section with two remarks.

\begin{remark}
\label{r7} The original example of Dieudonn\'{e}~\cite{Die} provides a Banach
$C(K)$-module $X$ with two generators such that every cyclic subspace is
weakly sequentially complete but $X$ cannot be represented as the sum of
two cyclic subspaces.
\end{remark}

\begin{remark}
\label{r8} The example of the Banach $l^{\infty}$-module $l^{1} \oplus c_{0}$
shows that in Theorem~\ref{t1} we cannot substitute the condition that
\textbf{every} cyclic subspace is weakly sequentially complete by a weaker
condition that for some set of generators $x_{1}, \ldots, x_{n}$ of $X$, the
cyclic subspaces $X(x_{i}), i = 1, \ldots, n$, are weakly sequentially complete.
\end{remark}

\bigskip

\subsection{Boolean algebras of projections of finite multiplicity.}

\hfill\break

Throughout this section we will assume that $\mathcal{B}$ is a Bade complete
Boolean algebra of projections on the Banach space $X.$ We let $K$ denote the
hyperstonian Stone representation space of $\mathcal{B}.$ As in section $2$,
we will assume that both $X$ and $X^{\star}$ are modules over $C(K)$ and that
$X^{\star\star}$ is a module over $C(K)^{\star\star}.$

\begin{definition}
\label{d7} Let $\mathcal{B}$ be a Bade complete Boolean algebra of projections
on $X$. $\mathcal{B}$ is said to be of \textit{uniform multiplicity} $n,$ if
there exist a set of nonzero pairwise disjoint idempotents $\{e_{\alpha}\}$ in
$\mathcal{B}$ with $\sup e_{\alpha}=1$ such that for any $e_{\alpha}$ and for
any $e\in\mathcal{B}$, $e\leq e_{\alpha}$ the $C(K)$-module $eX$ has exactly
$n$ generators.
\end{definition}

We will need the following result of Rall~\cite{Ra1} (for a proof, see Lemma 2
in~\cite{Or2}).

\begin{lemma}
\label{l6} Let $\mathcal{B}$ be of uniform multiplicity one on $X$. Then $X$
may be represented as a Banach lattice with order continuous norm such that
$\mathcal{B}$ is the Boolean algebra of band projections on $X$.
\end{lemma}

Next, as in~\cite{KO} we can state the following corollary of Theorem \ref{t1}.

\begin{corollary}
\label{c1} Let $X$ be a Banach space and let $\mathcal{B}$ be a Bade complete
Boolean algebra of projections on $X$ that is of uniform multiplicity $n.$
Then conditions (1)-(4) of Theorem \ref{t1} and condition (5) $p\cdot
X^{\star\star}=X$ are equivalent.
\end{corollary}

\begin{proof}
We have that the implications $(1)\Longrightarrow(2)\Longrightarrow(3)$ and
$(3)\Longleftrightarrow(4)$ hold as in the proof of Theorem \ref{t1}. The
proof would be complete if we show $(3)\Longrightarrow(5)$ and
$(5)\Longrightarrow(1).$ The proof of these implications follow exactly as in
Lemma~\ref{l5} for $(3)\Longrightarrow(5)$ and as in Theorem \ref{t1} for
$(5)\Longrightarrow(1)$. Thus, provided that the case $n=1$ is settled, we can
finish the proof by induction. Therefore suppose $\mathcal{B}$ is of uniform
multiplicity one on $X.$ Then by Lemma~\ref{l6}, $X$ may be represented as a
Banach lattice with order continuous norm such that $\mathcal{B}$ corresponds
to the band projections of the lattice. Assume that $(3)$ holds and $X$ is not
a KB-space (i.e., $p\cdot X^{\star\star}=(X^{\star})_{n}^{\star}\neq X$ ). Then by
Theorem~\ref{t00}, $X$ contains a sublattice isomorphic to $c_{0}$. Suppose
that $\{x_{n}\}$ is the disjoint sequence in this sublattice that corresponds
to the standard basis of $c_{0}$ and let $\{e_{n}\}$ be the disjoint
idempotents in $\mathcal{B}$ such that each $e_{n}$ is the band projection on
the band generated by $x_{n}$ in $X.$ Let $u=\sum\frac{1}{2^{n}}x_{n}$ in $X$.
Consider the cyclic subspace (i.e. the band) $X(u)$ of $X$. Since
$e_{n}u=\frac{1}{2^{n}}x_{n}$ for each $n$ , the sublattice generated by the
positive disjoint sequence $\{x_{n}\}$ is in $X(u)$ and this contradicts
$(3).$ Therefore $X$ is a KB-space, that is $p\cdot X^{\star\star}%
=(X^{\star})_{n}^{\star}=X$ and $X$ is weakly sequentially complete. Hence both
$(3)\Longrightarrow(5)$ and $(5)\Longrightarrow(1)$ hold for $n=1.$
\end{proof}

\begin{definition}
\label{d8} A Bade complete Boolean algebra of projections $\mathcal{B}$ on $X$
is said to be of finite multiplicity on $X$ if there exists a collection of
disjoint idempotents $\{e_{\alpha}\}$ in $\mathcal{B}$ such that, for each
$\alpha,$ $e_{\alpha}X$ is $n_{\alpha}$-generated and $\sup e_{\alpha}=1.$
\end{definition}

We note that the collection $\{n_{\alpha}\}$ of positive integers need not be
bounded. Then by a well known result of Bade \cite[XVIII.3.8, p. 2267]{DS},
there exists a sequence of disjoint idempotents $\{e_{n}\}$ in $\mathcal{B}$
such that, for each $n,$ $\mathcal{B}$ is of uniform multiplicity $n$ on
$e_{n}X$ and $\sup e_{n}=1.$ Also the norm closure of the sum of the sequence
of the spaces $\{e_{n}X\}$ is equal to $X.$ In our next result we will show
that the conclusions of Corollary \ref{c1} extend to this case. In the proof
of the theorem we will use a vector version of a standard disjoint sequence
method \cite[Theorem 1.c.10, p. 23]{LT1}.

\begin{theorem}
\label{t2} Let $X$ be a Banach space and let $\mathcal{B}$ be a Bade complete
Boolean algebra of projections on $X$ such that $\mathcal{B}$ is of finite
multiplicity on $X.$ Then conditions (1)-(4) of Theorem \ref{t1} and condition
(5) $p\cdot X^{\star\star}=X$ are equivalent.
\end{theorem}

\begin{proof}
It is clear that $(1)\Longrightarrow(2)\Longrightarrow(3)\Longleftrightarrow
(4).$ Initially we will show $(3)\Longrightarrow(1).$ We will use the notation
in the discussion before the statement. Since $\mathcal{B}$ is of finite
multiplicity on $X,$ by Corollary \ref{c1}, for each $n,$ $e_{n}X$ is weakly
sequentially complete. Let $\chi_{n}=e_{1}+e_{2}+\ldots+e_{n}$ for each $n.$
Then $\chi_{n}X$ is weakly sequentially complete for each $n.$ Moreover
$\chi_{n}\uparrow1$ in $\mathcal{B}$ implies that%
\begin{equation}
\underset{n\rightarrow\infty}{\lim}\left\Vert \left(  1-\chi_{n}\right)
x\right\Vert =0 \tag{*}%
\end{equation}
for any $x\in X.$ Suppose $\{x_{i}\}$ is a weak Cauchy sequence in $X$ such
that
\begin{equation}
\underset{i\rightarrow\infty}{\lim}f(\chi_{n}x_{i})=0 \tag{**}%
\end{equation}
for each $n$ and for all $f\in X^{\star}$. Suppose, on the other hand,
$\{x_{i}\}$ does not converge to $0$ in the weak topology. Then, without loss
of generality, we may suppose that for some real linear functional $g$ on $X$
with $\left\Vert g\right\Vert =1$ and $\delta>0,$ we have that, for all $i,$%
\[
g(x_{i})\geq\delta\text{ and }\underset{i\rightarrow\infty}{\lim}g(x_{i}%
)\geq\delta>0\text{ .}%
\]
We define a sequence $\{Y_{i}\}_{i=0}^{\infty}$ of convex subsets of $X$ by
\[
Y_{i}:=co\{x_{k}:k=i+1,i+2,\ldots\}.
\]
Let '$cl$' denote norm closure of a set in $X$. Then $\left(  \ast\ast\right)
$ and duality imply that
\begin{equation}
0\in cl(\chi_{n}Y_{i}),\text{ for each }n\text{ and each }i\text{.} \tag{***}%
\end{equation}
(Note that $y\in Y_{i}$ may have several convex representations by elements of
$\{x_{k}:k=i+1,i+2,\ldots\}$. If we pick some $y\in Y_{i},$ it should be
understood that at the same time we designate and fix some convex
representation of $y$ by elements of $\{x_{k}:k=i+1,i+2,\ldots\}.$ From then
onwards our references to the convex representation of $y$, will always be to
the designated convex representation that we fixed. In particular, only
finitely many elements of $\{x_{k}:k=i+1,i+2,\ldots\}$ are involved in the
convex representation. Namely the elements of $\{x_{k}:k=i+1,i+2,\ldots\}$
that have a non-zero coefficient in the convex representation. One of the
elements in this finite subset of $\{x_{k}:k=i+1,i+2,\ldots\}$ has the largest
index. We will refer to this index as the largest index in the convex
combination of $y.$)

We will choose a sequence $\{y_{n}\}_{n=1}^{\infty}$ that has the following properties:

\begin{enumerate}
\item $y_{1}\in co\{x_{k}:k=1,2,\ldots,p_{1}-1\}\subset Y_{0}$ such that
$\left\Vert (1-\chi_{p_{1}})y_{1}\right\Vert <\frac{\delta}{2};$

\item $y_{n}\in co\{x_{k}:k=p_{n-1}+1,\ldots,p_{n}-1\}\subset Y_{p_{n - 1}}$
such that $\left\Vert \chi_{p_{n-1}}y_{n}\right\Vert <\frac{\delta}{2^{2n}}$
and $\left\Vert (1-\chi_{p_{n}})y_{n}\right\Vert <\frac{\delta}{2^{n}}$ for
each $n=2,\ldots$.
\end{enumerate}

\noindent Note that when chosen as above, $\{p_{n}\}$ is a subsequence of
$\{n\}$ such that $1<p_{1}$ and $p_{n-1}+1<p_{n}$ ( $n=2,3,\ldots$ ). Now to
see that the sequence can be chosen as above, take $y_{1}\in Y_{0}$ and take
$p_{1}$ strictly greater than the largest index in the convex combination of
$y_{1},$ such that $\left\Vert (1-\chi_{p_{1}})y_{1}\right\Vert <\frac{\delta
}{2}$ (use $\left(  \ast\right)  $). Then take $y_{2}\in Y_{p_{1}}$ such that
$\left\Vert \chi_{p_{1}}y_{2}\right\Vert <\frac{\delta}{2^{4}}$ (use $\left(
\ast\ast\ast\right)  $). Next take $p_{2}$ strictly greater than the largest
index in the convex combination of $y_{2},$ such that $\left\Vert
(1-\chi_{p_{2}})y_{2}\right\Vert <\frac{\delta}{2^{2}}$ (use $\left(
\ast\right)  $). Suppose that we chose $y_{1},\ldots,y_{n}$ ( $n\geq2$) and
$p_{1}<\ldots<p_{n}$ as required. Then take $y_{n+1}\in Y_{p_{n}}$ such that
$\left\Vert \chi_{p_{n}}y_{n+1}\right\Vert <\frac{\delta}{2^{2(n+1)}}$ (use
$\left(  \ast\ast\ast\right)  $). Also take $p_{n+1}$ strictly greater than
the largest index in the convex combination of $y_{n+1}$ such that $\left\Vert
\left(  1-\chi_{p_{n+1}}\right)  y_{n+1}\right\Vert <\frac{\delta}{2^{n+1}}$
(use $\left(  \ast\right)  $). This shows that we can choose inductively the
sequence $\left\{  y_{n}\right\}  $ with the stated properties.

As chosen $\left\{  y_{n}\right\}  $ is a weak Cauchy sequence in $X$ such
that, for each $f\in X^{\star},$
\[
\underset{n\rightarrow\infty}{\lim}f(y_{n})=\underset{n\rightarrow\infty}%
{\lim}f(x_{n}).
\]
This follows, since $y_{n+1}\in co\left\{  x_{k}:k=p_{n}+1,\ldots
p_{n+1}-1\right\}  $ and therefore all the indices in the convex combination
of $y_{n+1}$ are strictly greater than $n.$ Also%
\[
\delta\leq g(y_{n})
\]
for all $n=1,2,\ldots$.

Now we define a new sequence $\left\{  z_{n}\right\}  $ such that
\[
z_{1}=\chi_{p_{1}}y_{1}\text{ and }z_{n}=(\chi_{p_{n}}-\chi_{p_{n-1}})y_{n}%
\]
for each $n=2,\ldots$. Since%
\[
\left\Vert y_{n}-z_{n}\right\Vert =\left\Vert (1-\chi_{p_{n}})y_{n}%
+\chi_{p_{n-1}}y_{n}\right\Vert <\frac{\delta}{2^{n}}(1+\frac{1}{2^{n}})
\]
for each $n=2,\ldots$, $\left\{  z_{n}\right\}  $ is also a weak Cauchy
sequence such that%
\begin{equation}
\underset{n\rightarrow\infty}{\lim}f(z_{n})=\underset{n\rightarrow\infty}%
{\lim}f(y_{n})=\underset{n\rightarrow\infty}{\lim}f(x_{n}). \tag{4*}%
\end{equation}
Furthermore, we have from above%
\begin{equation}
\frac{\delta}{2}\leq g(z_{n}) \tag{5*}%
\end{equation}
for each $n=1,2,\ldots$. Note that by definition the elements of the sequence
$\{z_{n}\}$ are in the range of disjoint projections, that is if $m\neq n$
then $(\chi_{p_{m}}-\chi_{p_{m-1}})(\chi_{p_{n}}-\chi_{p_{n-1}})=0.$ This
means that $l^{\infty}$ may be considered as an isometric unital subalgebra of
$C(K)$ where the correspondence is given by $(\alpha_{n})\in l^{\infty
}\leftrightsquigarrow\left(  \sum\alpha_{n}(\chi_{p_{n}}-\chi_{p_{n-1}%
})\right)  \in C(K)$ \cite{KO}. For any $(\xi_{n})\in l^{1},$ there is
$(\alpha_{n})\in l^{\infty}$ with $|\alpha_{n}|=1$ for all $n$ such that
$(\alpha_{n}\xi_{n})=(|\xi_{n}|)$ and $(\overline{\alpha_{n}}|\xi_{n}%
|)=(\xi_{n}).$ From this it follows that, when we consider $\sum\xi_{n}%
z_{n}\in X$ and use the embedding of $l^{\infty}$ in $C(K),$ we have%
\[
\left\Vert \sum\xi_{n}z_{n}\right\Vert =\left\Vert \sum|\xi_{n}|z_{n}%
\right\Vert .
\]
Then, by (5*), we have%
\[
\frac{\delta}{2}\left(  \sum|\xi_{n}|\right)  \leq g\left(  \sum|\xi_{n}%
|z_{n}\right)  \leq\left\Vert \sum\xi_{n}z_{n}\right\Vert \leq\left(
\sup\left\Vert z_{n}\right\Vert \right)  \left(  \sum|\xi_{n}|\right)  .
\]
That is $X$ has a subspace isomorphic to $l^{1}$ and $\left\{  z_{n}\right\}
$ corresponds to the standard basis of $l^{1}$ in the subspace of $X$ that is
isomorphic to $l^{1}.$ But this means that there exists $f\in X^{\star}$ such
that when restricted to $\left\{  z_{n}\right\}  $ we have%
\[
f(z_{n})=(-1)^{n}%
\]
for each $n=1,2,\ldots$. This contradicts the fact that $\left\{
z_{n}\right\}  $ is a weak Cauchy sequence and that $\underset{n\rightarrow
\infty}{\lim}f(z_{n})$ exists. Therefore the real linear functional $g$ with
the stated properties on $\left\{  x_{i}\right\}  $ cannot exist. So if
$\left\{  x_{i}\right\}  $ is a weak Cauchy sequence such that $\left\{
\chi_{n}x_{i}\right\}  $ converges weakly to $0$ for each $n$, then $\left\{
x_{i}\right\}  $ also converges to $0$ weakly in $X.$

Now suppose that $\left\{  x_{n}\right\}  $ is a weak Cauchy sequence in $X$
such that $\left\{  e_{k}x_{n}\right\}  $ converges weakly to $u_{k}\in
e_{k}X$ for each $k$. Since $\left\{  x_{n}\right\}  $ is bounded we have that
the sequence $\{u_{k}\}$ is also bounded by the same constant. Consider the
cyclic subspace $X(w)$ of $X$ that is generated by $w=\sum\frac{1}{2^{k}}%
u_{k}.$ By the hypothesis $(3)$, $X(w)$ does not contain any copy of $c_{0}.$
Therefore when represented as a Banach lattice with the positive
quasi-interior point $w,$ $X(w)$ becomes a KB-space. Also $e_{k}w=\frac
{1}{2^{k}}u_{k}$ for each $k$ imply that $\{u_{k}\}$ is a positive disjoint
sequence in $X(w).$ Let%
\[
v_{k}=u_{1}+u_{2}+\ldots+u_{k}%
\]
for each $k.$ Moreover $\{\chi_{k}x_{n}\}$ converges weakly to $v_{k}$ in
$\chi_{k}X$ for each $k.$ It is clear that $\{v_{k}\}$ is a norm bounded
sequence. (It has the same bound as $\{x_{n}\}.$) So $\{v_{k}\}$ is an
increasing and bounded positive sequence in the KB-space $X(w).$ Whence there
exists $v=\sup v_{k}$ such that $\left\Vert v-v_{k}\right\Vert \longrightarrow
0$ in $X(w).$ Therefore $\chi_{k}v=v_{k}$ for each $k.$ That is $\{x_{n}-v\}$
is a weak Cauchy sequence such that $\{\chi_{k}\left(  x_{n}-v\right)  \}$
converges weakly to $0$ for each $k.$ Then by the first part of our proof
$\{x_{n}-v\}$ also converges weakly to $0$ in $X$. This means that $X$ is
weakly sequentially complete. That is $\left(  3\right)  \Longrightarrow(1).$

Now suppose $(5)$ holds. By Lemma~\ref{l3} part (2), we have that $(5)$ holds
on each closed submodule of $X.$ In particular, we have that $(5)$ holds for
each cyclic subspace of $X.$ That is $p\cdot X(x)^{\star\star}=X(x)$ for each
$x\in X(x).$ Since each cyclic subspace is represented as a Banach lattice, by
Lemma~\ref{l2}, each cyclic subspace is a KB-space. This means that
$(5)\Longrightarrow(3).$ Therefore $(5)\Longrightarrow(1).$

Conversely, suppose $(1)$ holds. In particular $\chi_{n}X$ is weakly
sequentially complete for each $n.$ Then, by Corollary~\ref{c1}, each
$\chi_{n}X$ satisfies $(5).$ That is with $(\chi_{n}X)^{\star\star}=(\chi
_{n}X)^{oo}=\chi_{n}\cdot X^{\star\star}$, we have $p\cdot(\chi_{n}\cdot
X^{\star\star})=\chi_{n}X$ for each $n.$ Now suppose $p\cdot x^{\star\star
}=x^{\star\star}$ for some $x^{\star\star}\in X^{\star\star}$. Then
\[
p\cdot(\chi_{n}\cdot x^{\star\star})=\chi_{n}\cdot(p\cdot x^{\star\star}%
)=\chi_{n}\cdot x^{\star\star}\in\chi_{n}X\subset X\tag{6*}
\]
for each $n.$ Furthermore, for each $x^{\star}\in X^{\star},$ $p\cdot
x^{\star\star}=x^{\star\star}$ implies $\mu_{x^{\star},x^{\star\star}}\in
C(K)_{\star}$ (Lemma \ref{l1}(3)). Whence, since $\chi_{n}\uparrow1$ in $C(K),$
we have $\{\chi_{n}\cdot x^{\star\star}\}$ converges to $x^{\star\star}$ in
$X^{\star\star}$ in the weak*-topology. By (6*), this means that $\{\chi
_{n}\cdot x^{\star\star}\}$ is a weak Cauchy sequence in $X.$ Therefore $(1)$
implies that $x^{\star\star}\in X$ and the proof of $(1)\Longrightarrow(5)$ is complete.
\end{proof}

\section{Examples.}

We will end the paper by presenting some examples of Banach spaces $X$ with a
Bade complete Boolean algebra of projections $\mathcal{B}$ defined on $X$ such
that $\mathcal{B}$ is of finite multiplicity on $X$ but $X$ is not finitely
generated. All the examples that we present are related to the example of
Dieudonn\'{e}~\cite{Die} mentioned in Remark~\ref{r7}. We need the following
definition that was given by Tzafriri~\cite{Tz1}. The definition is motivated
by the Dieudonn\'{e}'s example.

\begin{definition}
\label{d9} Let $\mathcal{B}$ be a Bade complete Boolean algebra of projections
on a Banach space $X.$ Suppose $\mathcal{B}$ is of uniform multiplicity $n$ on
$X.$ Then $\mathcal{B}$ has property $\mathcal{D}$ on $X$ if for any $x_{i}\in
X(i=1,\ldots,n)$, any $e\in\mathcal{B}\setminus\{0\}$, and any $p,\ 1\leq
p<n$, $eX$ is not equal to the sum of $eX(x_{1},\ldots,x_{p})$ and
$eX(x_{p+1},\ldots,x_{n})$.
\end{definition}

The example of Dieudonn\'{e} is a Banach $L^{\infty}[0,\gamma]$-module $X$
with two generators and $\mathcal{B}[0,\gamma]$ is Bade complete on it . Here
$L^{\infty}[0,\gamma]$ is the algebra of equivalence classes of bounded
Lebesgue measurable functions on the interval $[0,\gamma]$ and $\mathcal{B}%
[0,\gamma]$ is the Boolean algebra of idempotents in $L^{\infty}[0,\gamma].$
The length of the interval, $\gamma$, is as computed by Dieudonn\'{e}.
Dieudonn\'e shows that $\mathcal{B}[0,\gamma]$ has property $\mathcal{D}$ on
$X$ for $n=2.$ Given any $n\geq2,$ it is clear that by following
Dieudonn\'{e}'s procedure one may construct a Banach $L^{\infty}[0,\gamma
]$-module $X$ with $n$-generators such that $\mathcal{B}[0,\gamma]$ has
property $\mathcal{D}$ on $X$ for $n$. Also as mentioned in Remark~\ref{r7}
all these spaces are weakly sequentially complete since they are constructed
as closed subspaces of $KB$-spaces.

\begin{example}
(1) Let $X_{2}$ denote the original space of Dieudonn\'{e} in Remark~\ref{r7}.
That is, it is a Banach $L^{\infty}[0,\gamma]$-module with two generators and
$\mathcal{B}[0,\gamma]$ is Bade complete on $X_{2}$ with property
$\mathcal{D}$ for $n=2$. $\mathcal{B}[0,\gamma]$ is the Boolean algebra of the
characteristic functions of the Lebesgue measurable sets with positive measure
in $[0,\gamma].$ Similarly let $X_{n}$ denote the Banach $L^{\infty}%
[0,\gamma]$-module with $n$-generators such that $\mathcal{B}[0,\gamma]$ is
Bade complete on $X_{n}$ with property $\mathcal{D}$ for $n$ (constructed by
following Dieudonn\'{e}'s method as remarked above). Now consider
$l^{p}(\{X_{n}\})$ as an $l^{\infty}(\{L^{\infty}[0,\gamma]\})$-module where
$1\leq p<\infty$. It is to be understood that $\{x_{n}\}\in l^{p}(\{X_{n}\})$
means $x_{n}\in X_{n}$ for each $n$ and $\{||x_{n}||\}\in l^{p}$. Similar
understanding holds for the algebra $l^{\infty}(\{L^{\infty}[0,\gamma]\}).$
Let $\mathcal{B}$ be the Boolean algebra of the idempotents in $l^{\infty
}(\{L^{\infty}[0,\gamma]\}).$ Then $\mathcal{B}$ is of finite multiplicity on
$l^{p}(\{X_{n}\}).$ But $l^{p}(\{X_{n}\})$ is infinitely generated.

(2) Let $\Gamma$ be an uncountable set. Let $\{n_{\alpha}\}_{\alpha\in\Gamma}$
be an unbounded subset of $\mathbb{N}$. Now consider $l^{p}(\{X_{n_{\alpha}%
}\})$ as an $l^{\infty}(\{L^{\infty}[0,\gamma]_{n_{\alpha}}\})$-module as in
(1). Then once more $\mathcal{B}$ is of finite multiplicity on $l^{p}%
(\{X_{n_{\alpha}}\}).$ Moreover by Bade's theorem there exists a sequence
$\{e_{n}\}$ in $\mathcal{B}$ such that $\mathcal{B}$ is of uniform
multiplicity $n$ on $e_{n}l^{p}(\{X_{n_{\alpha}}\})$ and $\sup e_{n}=1$. Since
$\Gamma$ is uncountable some of the $e_{n}l^{p}(\{X_{n_{\alpha}}\})$ are
necessarily infinitely generated.

(3) Let $Y_{2}$ be the reflexive space in Example 3 of \cite{KO}. Since
$Y_{2}$ is constructed by altering $X_{2}$, one can construct similarly
$Y_{n}$ from $X_{n}$ above with the same properties. However, in addition,
$Y_{n}$ would be reflexive. So we can consider $l^{p}(\{Y_{n}\})$ and
$l^{p}(\{Y_{n_{\alpha}}\})$ for $1\leq p<\infty$. These spaces will have the
same multiplicity properties as their counterparts in (1) and (2).

Now all the spaces in (1), (2), (3) are weakly sequentially complete. But the
ones in (3) are, in addition, reflexive for $1<p<\infty.$ To see this,
initially consider an example as in (1). Take $\{x_{n}\}\in l^{p}(\{X_{n}\}).$
Let $X_{n}(x_{n})$ be the cyclic subspace generated by $x_{n}$ in $X_{n}$. The
space $X_{n}$ is weakly sequentially complete and $n$-generated. Hence, by
Theorem~\ref{t1}, as a Banach lattice, the cyclic subspace $X_{n}(x_{n})$ is a
$KB$-space. Now the cyclic subspace generated by $\{x_{n}\}$ in $l^{p}%
(\{X_{n}\})$ is given by the Banach lattice $l^{p}(\{X_{n}(x_{n})\})$. It is
immediate that $l^{p}(\{X_{n}(x_{n})\})$ is a $KB$-space and hence it is
weakly sequentially complete. Then, by Theorem~\ref{t2}, $l^{p}(\{X_{n}\})$ is
weakly sequentially complete. For examples in (2), suppose that
$\{x_{n_{\alpha}}\}\in l^{p}(\{X_{n_{\alpha}}\})$. Then $x_{n_{\alpha}}\neq0$
for at most countably many $\alpha\in\Gamma$. Whence, the argument given for
the examples in (1) shows that $l^{p}(\{X_{n_{\alpha}}\})$ is weakly
sequentially complete. This also covers the weak sequential completeness for
all the examples in (3). When $1<p<\infty$, to see that the spaces in (3) are
reflexive, it is sufficient to calculate the duals directly.
\end{example}

\end{document}